\documentclass[11pt]{amsart}
\usepackage{mathtools, amssymb, amsfonts, amsthm, tikz, bm, enumitem}
\usepackage{color, graphicx}
\usepackage{amsfonts}

\newtheorem{thm}{Theorem}[section]
\newtheorem{lemma}[thm]{Lemma}

\newtheorem{prop}[thm]{Proposition}

\newtheorem{defi}[thm]{Definition}
\newtheorem{eg}[thm]{Example}

\theoremstyle{plain}
\newtheorem{theo}[thm]{Theorem}

\theoremstyle{definition}

\newtheorem{rem}[thm]{Remark}

\numberwithin{equation}{section}

\def\sq{\square}

\def\zz{\mathbb Z}
\def\nn{\mathbb N}

\def\rr{\mathbb R}

\def\ov{\overline}

\def\la{\lambda}

\def\al{\alpha}
\def\be{\beta}

\def\T{\mathbf{T}}

\def\wt{\widetilde}
\def\<{\langle}
\def\>{\rangle}

\def\Z{ {\text {\rm Z} } }

\def\U{{\text {\rm U} } }

\def\0{{\mathbf 0}}

\def\.{\hskip.06cm}
\def\ts{\hskip.03cm}

\def\bx{{\textbf{x}}}
\def\bt{{\textbf{t}}}
\def\by{{\textbf{y}}}

\def\Z{\mathbb{Z}}
\def\R{\mathbb{R}}
\def\N{\mathbb{N}}
\def\QQ{\mathbb{Q}}

\def\T{\mathcal{T}}

\def\LS{\mathcal{S}}

\def\W{\mathcal{W}}
\def\U{\mathcal{U}}

\newcommand{\ex}{\exists\ts}
\renewcommand{\for}{\forall\ts}

\newcommand{\cj}[1]{\overline{#1}}
\newcommand{\n}{\cj{n}}
\renewcommand{\b}{\cj{b}}

\def\a{\cj{a}}
\def\c{\cj{c}}
\def\d{\cj{d}}

\newcommand{\x}{\mathbf{x}}
\renewcommand{\t}{\mathbf{t}}

\def\v{\cj v}

\newcommand{\y}{\mathbf{y}}
\newcommand{\z}{\mathbf{z}}

\newcommand{\floor}[1]{\lfloor#1\rfloor}

\def\polyin{\textup{poly}}

\def\nin{\noindent}

\def\proj{\textup{proj}}

\renewcommand\L{\mathcal{L}}

\def\NP{{\textup{\textsf{NP}}}}

\def\sg{\textup{Sg}}

\def\proj{\textup{proj}}
\def\iproj{\textup{proj}^{-1}}

\def\affdim{\textup{dim}}
\def\rank{\textup{rank}}

\def\cpl{\backslash}
\def\Pr{\textup{PA}}
\def\span{\textup{span}}
\def\wh{\widehat}
\def\w{\cj{w}}
\def\pat{\bm}

\def\phi{\varphi}

\title{Enumerating projections of integer points in unbounded polyhedra}

\author[Danny Nguyen \and Igor Pak]{Danny Nguyen$^{\star}$ \and Igor~Pak$^{\star}$}

\thanks{\thinspace ${\hspace{-.45ex}}^\star$Department of Mathematics,
UCLA, Los Angeles, CA, 90095.
\hskip.06cm
Email:
\hskip.06cm
\texttt{\{ldnguyen,\ts{pak}\}@math.ucla.edu}}

\thanks{
\today}

\begin{document}

\maketitle

\begin{abstract}
We extend the \emph{Barvinok--Woods algorithm} for enumerating projections of integer points in polytopes to unbounded polyhedra.  For this, we obtain a new
structural result on projections of \emph{semilinear subsets} of the integer lattice.
We extend the results to general formulas in \emph{Presburger Arithmetic}. We also
give an application to the \emph{$k$-Frobenius problem}.
\end{abstract}

\section{Introduction}\label{sec:Introduction}

\subsection{The results}
\emph{Integer linear programming} in fixed dimension is a classical
subject~\cite{L}.  The pioneering result by Lenstra~\cite{L} shows that the
\emph{feasibility} of integer linear programming in a fixed dimension $n$ can
be decided in polynomial time:
$$(\circ) \qquad
\exists \ts\bx \in \zz^{n} \, : \, A\ts\bx \, \le \, \ov b\ts.
$$
Here $A \in \Z^{d \times n}$ and $\b \in \Z^{d}$ are the input, and Lenstra's algorithm runs in polynomial time compared to their total bit length.
This result was extended by Kannan~\cite{K1}, who showed that
\emph{parametric integer linear programming} in fixed dimensions can
be decided in polynomial time:
$$(\circ\circ) \qquad
\forall \ts\by \in P\cap\zz^{n}  \; \; \exists \ts\bx \in \zz^k \, : \, A\ts\bx \. + \. B\ts\by \, \le \, \ov b\ts.
$$
Both results rely on difficult results in geometry of numbers and can be viewed
geometrically: $(\circ)$ asks whether a polyhedron $\ts Q = \{A\x \le \b\ts\} \subseteq \rr^n$ \ts
has an integer point. Similarly, $(\circ\circ)$ asks whether every integer point in the polyhedron $P = \{C\x \le \d \ts\} \subseteq \R^{n}$ is the projection of an integer point in the polyhedron $\ts Q = \{A\x +B\y \le \b\ts\} \subseteq \rr^{m}$, where $m=n+k$.

Barvinok~\cite{B1} famously showed that the number of integer points in polytopes
in a fixed dimension $n$ can be computed in polynomial time. He used a technology of
\emph{short generating functions} (GFs) to enumerate the integer points in general
(possibly unbounded) rational polyhedra in $\rr^d$ in the following form:
$$(\divideontimes) \qquad f(\bt) \, = \,
\sum_{i=1}^N \, \frac{c_{i} \. \bt^{\a_i}}{(1-\bt^{\b_{i\ts 1}})\cdots (1-\bt^{\b_{i\ts k_{i}}})}\ts,
$$
where $c_{i} \in \mathbb{Q}, \; \cj a_{i}, \cj b_{ij} \in \zz^{n}$ and $\ts\bt^{\a} = t_1^{a_{1}}\cdots {}\ts t_n^{a_{n}}\ts$ if $\ts\a = (a_{1},\ldots,a_{n}) \in \zz^{n}$.
Under the substitution $\t \gets 1$ in $(\divideontimes)$, one can count the number of integer points in a (bounded) polytope $Q$, and thus solves $(\circ)$ quantitatively for the bounded case.
In general, one can also succinctly represent integer points in the intersections, unions and complements of general (possibly unbounded) rational polyhedra~\cite{B3,BP} in $\R^{n}$ using short GFs.

Barvinok's algorithm was extended to count projections of integer points in bounded polytopes by Barvinok
and Woods \cite{BW}, see Theorem~\ref{th:BW}.  The result has a major technical drawback:
while it does generalize Kannan's result for bounded $P$ and $Q$ as in $(\circ\circ)$,
it does not apply for unbounded polyhedra.  The main result of this paper is
an extension of the Barvinok--Woods algorithm to the unbounded case (Theorem~\ref{th:main_2}).

\begin{theo}\label{th:main_2}
Let $m, n \in \nn$ be fixed dimensions.
Given a $($possibly unbounded$\ts)$ polyhedron $Q = \{\x \in \R^{m} : A\x \le \b \}$
and an integer linear transformation $T : \rr^{m} \to \rr^{n}$
which satisfies $T(Q) \subseteq \rr_{+}^{n}$,
let $g(\t)$ be the GF for $T(Q \cap \Z^{m})$, i.e.,
$$
g(\t) \, = \, \sum_{\y \; \in \; T(Q \cap \, \Z^{m})} \. \t^{\y}\..
$$
Then there is a polynomial time algorithm to compute $g(\t)$ in the form of
a short GF~$(\divideontimes)$.
\end{theo}

Here by an integer linear transformation we mean that the linear map $T$ is presented by a matrix $T \in \zz^{n \times m}$.
To illustrate our theorem, consider:

\begin{eg}
{\rm
Let $Q = \{(x,y,z) \in \R_{+}^{3} : x = 2 y + 5 z\}$ and $T$ be the projection from $\zz^{3}$ onto the first coordinate $\zz^{1}$. Then $T(Q \cap \zz^{3})$ has a short GF:
\begin{equation*}
\;\frac{1}{(1-t^{2})(1-t^{5})} - \frac{t^{10}}{(1-t^{2})(1-t^{5})} = 1 + t^{2} + t^{4} + t^{5} + t^{7} + \ldots\.
\end{equation*}
}
\end{eg}

Our main tool is a structural result describing projections of \emph{semilinear sets},
which are defined as disjoint union of intersections of polyhedra and lattice cosets.
More precisely, we prove that such projections are also semilinear and give bound on
(combinatorial) complexity of the projections (Theorem~\ref{th:main_1}). In combination
with the Barvinok--Woods theorem this gives the extension to unbounded polyhedra.

We then present a far-reaching generalization of our results to all formulas in
\emph{Presburger Arithmetic}: we first prove a structural result
(Theorem~\ref{th:pres}) and then a generalization of Theorem~\ref{th:main_2}
(Theorem~\ref{th:pres_GF}).  We illustrate the power of our generalization
in the case of the \emph{$k$-Frobenius Problem}.

\subsection{Connections and applications}
After Lenstra's algorithm, many other methods for fast integer programming in fixed dimensions have been found (see~\cite{E1,FT}).
Kannan's algorithm was strengthened in~\cite{ES}.
Barvinok's algorithm has been simplified and improved in~\cite{DK,KV}.
Both Barvinok's and Barvinok--Woods' algorithms have been
implemented and used for practical computation \cite{DHTY,Kop,VSBLB}.

Let us emphasize that in the context of parametric integer programming,
there are two main reasons to study unbounded polyhedra:

\medskip

\nin
{\tt {\small (1)}} \ts  Working with short GFs of integer points
in unbounded polyhedra allows to compute to various integral
sums and valuations over convex polyhedra.  We refer to
\cite{B+,B3,BV} for many examples and further references.

\smallskip

\nin
{\tt {\small (2)}} \ts  For a fixed unbounded polyhedron $Q \subseteq \R^{m}$ and a varying polytope $P \subset \R^{n}$ in
$(\circ\circ)$, one can count the number of points in the projection of~$Q \cap \Z^{m}$ within~$P$.
This is done by intersecting~$Q$ with a box of growing size and then projecting it.
The Barvinok--Woods algorithm is called multiple times for different boxes, which depend on $P$.
Our approach allows one to call the Barvinok--Woods algorithm only once to project~$Q \cap \Z^{m}$
(unbounded), and then call a more economical Barvinok's algorithm to compute
the intersection with~$P$.  See Section~\ref{s:k-feas} for an explicit example.

In conclusion, let us mention that semilinear sets are well studied subjects in both computer science and logic.
The fact that the category of semilinear sets are closed under taking projections is not new.
Ginsburg and Spanier \cite{GS} showed that semilinear sets are exactly those sets definable in Presburger arithmetic, which are closed under Boolean operations and projections.
Woods \cite{W15} also characterized semilinear sets as exactly those sets with rational generating functions, which also implies closedness under Boolean operations and projections.
In our paper, we prove the structural result on projections
of semilinear sets by a direct argument, without using tools from logic (e.g. quantifier elimination).
By doing so, we obtain effective polynomial
bounds for the number of polyhedral pieces and the
facet complexity of each piece in the projection.

\bigskip

\section{Standard definitions and notations}\label{sec:Notations}

\nin
We use $\nn \ts = \ts \{0,1,2,\ldots\}$, $\Z_{+} = \{1,2,\ldots\}$ and $\R_{+} = \{x \in \R : x \ge 0\}$.

\nin
All constant vectors are denoted $\a, \b, \c, \d, \n, \v,$ etc.

\nin
Integer matrices are denoted $A, B, C$, etc.

\nin
Variables are denoted $x,y,z$, etc.; vectors of variables are denoted $\x, \y, \z$, etc.

\nin
We write $\x \le \y$ if $x_{j} \le y_{j}$ for coordinate in vectors $\x$ and $\y$.

\nin
We also write $\x \le N$ to mean that each coordinate is~$\le N$.





\nin
GF is an abbreviation for ``\emph{generating function}.''

\nin
Multivariate GFs are denoted by $f(\t), g(\t), h(\t)$, etc.

\nin
A \emph{polyhedron} is an intersection of finitely many closed half-spaces in~$\rr^n$.



\nin
A \emph{polytope} is a bounded polyhedron.

\nin
Polyhedra and polytopes are denoted by $P, Q, R$, etc.

\nin
The \emph{affine dimension} of $P$ is denoted by $\affdim(P)$.

\nin
Integer lattices are denoted by $\L,\T,\U,\W$, etc.

\nin
Let $\rank(\L)$ denotes the \emph{rank} of lattice~$\L$.

\nin
\emph{Patterns} are denoted by $\pat L, \pat T, \pat S, \pat U, \pat W$, etc.

\nin
Let $\phi(\cdot)$ denotes the \emph{binary length} of a number, vector, matrix, GF, or a logical formula.

\nin
For a polyhedron $Q$ described by a linear system $A\x \le \b$, let $\phi(Q)$ denote
the total length
$\phi(A)+\phi(\b)$.

\nin
For a lattice $\L$ generated by a matrix $A$, we use $\phi(\L)$ to denote $\phi(A)$.

\bigskip

\section{Structure of a projection}\label{sec:Structure}

\subsection{Semilinear sets and their projections}

In this section, we assume all dimensions $m,n$, etc., are fixed. We emphasize that all lattices mentioned are of full rank. All inputs are in binary.

\begin{defi}\label{def:proj}
{\rm
Given a set $X \subseteq \R^{n+1}$, the \emph{projection} of $X$ onto $\R^{n}$, denoted by $\proj(X)$, is defined as
\[
\proj(X) \coloneqq \{(x_{2},\dots,x_{n}) : (x_{1},x_{2},\dots,x_{n+1}) \in X\} \subseteq \R^{n}.
\]
For any $\y \in \proj(X)$, denote by $\iproj(\y) \subseteq X$ the preimage of $\y$ in $X$.
}
\end{defi}

\begin{defi}\label{def:pattern}
{\rm
Let $\L \subseteq \Z^{n}$ be a full-rank lattice.
A \emph{pattern $\pat L$ with period $\L$} is a union of finitely many (integer) cosets of $\L$.
For any other lattice $\L'$, if $\pat L$ can be expressed as a finite union of cosets of $\L'$, then we also call $\L'$ a period of $\pat L$.

Given a rational polyhedron $Q$ and a pattern $\pat L$, the set $Q \cap \pat L$ is called a \emph{patterned polyhedron}. When the pattern $\pat L$ is not emphasized, we simply call $Q$ a \emph{patterned polyhedron with period $\L$}.
}
\end{defi}

\begin{defi}\label{def:sem_lin}
{\rm
A \emph{semilinear} set $X$ is a set of the form
\begin{equation}\label{eq:sem_lin}
X = \bigsqcup_{i=1}^{k} \, Q_{i} \cap \pat L_{i}\.,
\end{equation}
where each $Q_{i} \cap \pat L_{i}$ is a patterned polyhedron with period $\L_{i}$,
and the polyhedra $Q_{i}$ are pairwise disjoint.
The \emph{period length} $\psi(X)$ of $X$ is defined as
\begin{equation*}
\psi(X) = \sum_{i=1}^{k} \phi(Q_{i}) + \phi(\L_{i}).
\end{equation*}
Note that $\psi(X)$ does not depend on the number of cosets in each $\pat L_{i}$.
Define
\[
\eta(X) \coloneqq \sum_{i=1}^{k} \eta(Q_{i}),
\]
where each $\eta(Q_{i})$ is the number of facets of the polyhedron $Q_{i}$.
}
\end{defi}

\begin{rem}\label{rem:semilinear}
In Theoretical~CS
literature, semilinear sets are often explicitly presented as a finite union of \emph{linear sets}.
Each linear set is a translated semigroup generated by a finite set of vectors in $\zz^{n}$.
This explicit representation by generators makes operations like projections easy to compute, while structural
properties harder to establish (see e.g.~\cite{CH} and the references therein).
The equivalence of the two representations is proved in~\cite{GS}.
\end{rem}

Our main structural result is the following theorem.

\begin{theo}\label{th:main_1}
Let $m \in \nn$ be fixed.
Let $X \subseteq \Z^{m}$ be a semilinear set of the form \eqref{eq:sem_lin}. Let $T: \R^{m} \to \R^{n}$ be a linear map satisfying $T(\Z^{m}) \subseteq \Z^{n}$.
Then $T(X)$ is also a semilinear set, and there exists a decomposition
\begin{equation}\label{eq:decomp_many}
T(X) \, = \, \bigsqcup_{j=1}^{r} \, R_{j} \cap \pat T_{j} \.,
\end{equation}
where each $R_{j} \cap \pat T_{j}$ is a patterned polyhedron in $\R^{n}$ with period $\T_{j} \subseteq \Z^{n}$. The polyhedra $R_{j}$ and lattices $\T_{j}$ can be found in time $\polyin(\psi(X))$.
Moreover,
\[
r = \eta(X)^{O(m!)} \quad \text{and} \quad \eta(R_{j}) = \eta(X)^{O(m!)}, \; 1 \le j \le r.
\]
\end{theo}

\begin{rem}{\rm
The above result describes all pieces $R_{j}$ and periods $\T_{j}$ in polynomial time. However, it does not explicitly describe the patterns $\pat T_{j}$.
The latter is actually an $\NP$-hard problem (see Remark~\ref{rem:NP-hard}).
}
\end{rem}

\begin{rem}{\rm
In the special case when $X$ is just one polyhedron $Q \cap \Z^{m}$, the first piece $R_{1} \cap \pat T_{1}$ in \eqref{eq:decomp_many} has a simple structure.
Theorem~1.7 in \cite{AOW} identifies and describes  $R_{1} \cap \pat T_{1}$ as $R_{1} = T(Q)_{\gamma}$ and $\pat T_{1} = T(\Z^{m})$.
Here $T(Q)_{\gamma}$ is the \emph{$\gamma$-inscribed} polyhedron inside $T(Q)$ (see \cite[Def.~1.6]{AOW}).
However, their result does not characterize the remaining pieces $R_{j} \cap \pat T_{j}$ in the projection $T(X)$.
Thus, Theorem~\ref{th:main_1} can also be seen as a generalization of the result in \cite{AOW} to semilinear sets, with a complete description of the projection.
}
\end{rem}

For the proof of Theorem~\ref{th:main_1}, we need a technical lemma:

\begin{lemma}\label{lem:proj}
Let $n \in \nn$ be fixed.
Consider a patterned polyhedron $(Q \cap \pat L) \subseteq \R^{n+1}$ with period $\L$. There exists a decomposition
\begin{equation}\label{eq:decomp}
\proj(Q \cap \pat L) = \bigsqcup_{j=0}^{r} \, R_{j} \cap \pat T_{j} \.,
\end{equation}
where each $R_{j} \cap \pat T_{j}$ is a patterned polyhedron in $\R^{n}$ with period $\T_{j} \subseteq \Z^{n}$. The polyhedra $R_{j}$ and lattices $\T_{j}$ can be found in time $\polyin(\phi(Q) + \phi(\L))$. Moreover,
\[
r \ts = \ts O\bigl(\eta(Q)^{2}\bigr) \quad \text{and} \quad \eta(R_{j}) \ts = \ts O\bigl(\eta(Q)^2\bigr), \quad \text{for all} \ \ 0 \le j \le r\ts.
\]
\end{lemma}

We postpone the proof of the lemma until Subsection~\ref{ss:lemma-proj-proof}.

\subsection{Proof of Theorem~\ref{th:main_1}.}\label{sec:short_proof}
We begin with the following definitions and notation.

\begin{defi}\label{def:copoly}
{\rm
A \emph{copolyhedron} $P \subseteq \R^{d}$ is a polyhedron with possibly some open facets.
If $P$ is a rational copolyhedron, we denote by $\floor{P}$ the (closed) polyhedron obtained from $P$ by sharpening each open facet $(\a \. \x < b)$ of $P$ to $(\a \. \x \le b - 1)$, after scaling $\a$ and $b$ to integers. Clearly, we have $P \cap \Z^{d} = \floor{P} \cap \Z^{d}$.
}
\end{defi}

WLOG, we can assume $n \le m$ and the linear map $T : \R^{m} \to \R^{n}$ has $\rank(T) = n$.
Also denote by $T$ the integer matrix in $\zz^{n \times m}$ representing this linear map.
We can rearrange the coordinates in $\R^{m}$ so that the first $n$ columns in $T$ form a non-singular minor.

Recall that $X$ has the form \eqref{eq:sem_lin} with each $Q_{i} \cap \pat L_{i}$ having period $\L_{i}$.
For each $i$, define the polyhedron
\begin{equation}\label{eq:Q_hat_def}
\wh Q_{i} \coloneqq \bigl\{ (\x,\y) : \y = T\x \; \text{ and } \; \x \in Q_{i} \bigr\} \subseteq \R^{m+n}.
\end{equation}
Consider the pattern $\pat U_{i} = \pat L_{i} \oplus \Z^{n} \subseteq \Z^{m+n}$ with period $\U_{i} = \L_{i} \oplus \Z^{n}$.
Then $\wh Q_{i} \cap \pat U_{i}$ is a patterned polyhedron in $\R^{m+n}$ with period $\U_{i}$.
Define the projection $S : \R^{m+n} \to \R^{n}$ with $S(\x,\y) = \y$.
By \eqref{eq:Q_hat_def}, we have:
\begin{equation*}
T(Q_{i} \cap \pat L_{i}) = S(\wh Q_{i} \cap \pat U_{i}) \quad \text{and} \quad T(X) = S \left( \bigsqcup_{i=1}^{r} \. \wh Q_{i} \cap \pat U_{i} \right) =\bigcup_{i=1}^{r} S(\wh Q_{i} \cap \pat U_{i}),
\end{equation*}
We can represent $S = S_{m} \circ \dots \circ S_{1}$, where each $S_{i} : \R^{m+n-i+1} \to \R^{m+n-i}$ is a projection along the $x_{i}$ coordinate.

Let $H \subset \rr^{m+n}$ be the subspace defined by $\y = T\x$.
First, we show that the initial $n$ projections $F = S_{n} \circ \dots \circ S_{1}$ are injective on $H$.
Indeed, assume $(\x,\y),\, (\x',\y')$ are two points in $H$ with $F(\x,\y) = F(\x',\y')$.
Since $F$ projects along the first $n$ coordinates of $\x$ and $\x'$, we have $(x_{n+1}, \dots, x_{m},\.\y) = (x'_{n+1}, \dots, x'_{m},\.\y')$.
Thus, $\y = \y'$, which implies $T \x = T \x'$.
Let $B \in \Z^{n \times n}$ be the first $n$ columns in $T$, which forms a non-singular minor as assumed earlier.
Since $T \x = T \x'$ and $(x_{n+1},\dots,x_{m}) = (x'_{n+1},\dots,x'_{m})$, we have $B\.(x_{1},\dots,x_{n}) = B\.(x'_{1},\dots,x'_{n})$.
This implies $(x_{1},\dots,x_{n}) = (x'_{1},\dots,x'_{n})$.
We conclude that $(\x,\y) = (\x',\y')$, and $F$ is injective on $H$.

By \eqref{eq:Q_hat_def}, we have $\wh Q_{i} \cap \pat U_{i} \subseteq H$ for every $i$.
Because $F : \R^{m+n} \to \R^{m}$ is injective on $H$, the semilinear structure of $\bigl( \. \bigsqcup \. \wh Q_{i} \cap \pat U_{i} \. \bigr)$ is preserved by $F$.
For convenience, we also denote by $\bigl( \. \bigsqcup \. \wh Q_{i} \cap \pat U_{i} \. \bigr)$ the semilinear set after applying $F$, which is now a subset of $\Z^{m}$.
Now we repeatedly apply Lemma~\ref{lem:proj} to the remaining projections $S_{m} \circ \dots \circ S_{n+1}$.
Starting with the projection $S_{n+1}$ applied on each piece $Q_{i} \cap \pat U_{i} \subseteq \zz^{m}$, we get:
\begin{equation}\label{eq:partial}
S_{n+1}(\wh Q_{i} \cap \pat U_{i}) \, = \, \bigsqcup_{j=0}^{r_{i}} \. R_{ij} \cap \pat T_{ij} \quad \; \text{for} \ \ 1 \le i \le k\ts,\, 1 \le j \le r_{i}\ts,
\end{equation}
where each $R_{ij} \cap \pat T_{ij}$ is a patterned polyhedron in $\Z^{m-1}$ with period $\T_{ij}$.
Note that two polyhedra $R_{ij}$ and $R_{i'j'}$ can be overlapping if $i \neq i'$.
However, we can refine all $R_{ij}$ into polynomially many disjoint copolyhedra $P_{1},\dots,P_{e} \subseteq \R^{m-1}\.$, so that
\begin{equation}\label{eq:refine_pieces}
\bigcup_{i=1}^{k} \. \bigcup_{j=1}^{r_{i}} \. R_{ij} \, = \, \bigsqcup_{d=1}^{e} P_{d}\..
\end{equation}
For each $P_{d}$, there is a pattern $\pat W_{d}$ with period $\W_{d} \subseteq \Z^{m-1}$ which fits with those $\pat T_{ij}$ for which $P_{d} \subseteq R_{ij}$.
The (full-rank) period  $\W_{d}$ can simply be taken as the intersection of polynomially many (full-rank) periods $\T_{ij}$ for which $P_{d} \subseteq R_{ij}$.
Taking intersections of lattices in a fixed dimension can be done in polynomial time using Hermite Normal Form (see~\cite{KB}).
We also round each $P_{d}$ to $\floor{P_{d}}$ (see Definition~\ref{def:copoly}).
From~\eqref{eq:partial} and~\eqref{eq:refine_pieces} we have:
\begin{equation*}
S_{n+1} \Big( \bigsqcup_{i=1}^{k} \wh Q_{i} \cap \pat U_{i} \Big) \, = \, \bigcup_{i=1}^{k} S_{n+1}(\wh Q_{i} \cap \pat U_{i}) \, = \, \bigsqcup_{d=1}^{e} \; \floor{P_{d}} \cap \pat W_{d} \..
\end{equation*}
The above RHS is a semilinear set in $\Z^{m-1}$.
A similar argument applies to $S_{m} \circ \dots \circ S_{n+2}$.
In the end, we have a semilinear decomposition for $T(X) \subseteq \Z^{n}$, as in \eqref{eq:decomp_many}.

Using Lemma~\ref{lem:proj}, we can bound the number of polyhedra $r_{i}$ in \eqref{eq:partial}, and also
the number of facets $\eta(R_{ij})$ for each $R_{ij}$.
It is well known that any $q$ hyperplanes in $\R^{m}$ partition the space into at most $O(q^{m})$ polyhedral regions.
This gives us a polynomial bound on $e$, the number of refined pieces in \eqref{eq:refine_pieces}.
By a careful analysis, after $m$ projections, the total number $r$ of pieces in the final decomposition \eqref{eq:decomp_many} is at most $\eta(X)^{O(m!)}$.
Each piece $R_{j}$ also has at most $\eta(X)^{O(m!)}$ facets.
\ $\sq$

\subsection{Proof of Lemma~\ref{lem:proj}.} \label{ss:lemma-proj-proof}

The proof is by induction on $n$. The case $n=0$ is trivial.
For the rest of the proof, assume $n \ge 1$.

Let $\pat L \subseteq \Z^{n+1}$ be a full-rank pattern with period $\L$ as in the lemma.
Then, the projection of $\pat L$ onto $\Z^{n}$ is another pattern $\pat L'$ with full-rank period $\L' = \proj(\L)$.\footnote{Here a basis for $\L'$ can be computed in polynomial time by applying Hermite Normal Form to a basis of $\L$, whose first coordinates $x_{1}$ should be set to $0$.}
Since $\L$ is of full rank, we can define
\begin{equation}\label{eq:ldef}
\ell = \min \{ t \in \Z_{+} : (t,0,\dots,0) \in \L\}.
\end{equation}

Let $R = \proj(Q)$.
Assume $Q$ is described by the system $A\x \le \b$.
Recall the \emph{Fourier--Motzkin elimination method}  (see  \cite[\S 12.2]{Schrijver}), which gives the facets of~$R$ from those of~$Q$.
First, rewrite and group the inequalities in $A\x \le \b$ into
\begin{equation}\label{eq:FM}
A_{1}\y + \b_{1} \le x_{1} ,\quad x_{1} \le A_{2}\y + \b_{2}  \quad \text{ and } \quad A_{3}\y \le \b_{3},
\end{equation}
where $\y = (x_{2},\dots,x_{n+1}) \in \R^{n}$.
Then $R$ is described by a system $C\y \le \d$, which consists of $(A_{3}\y \le \b_{3})$ and $(\a_{1} \y + b_{1} \le \a_{2} \y + b_{2})$ for every possible pair of rows $\a_{1} \y + b_{1}$ and $\a_{2} \y + b_{2}$ from the first two systems in \eqref{eq:FM}.

In case one of the two systems $A_{1}\y + \b_{1} \le x_{1}$ and $x_{1} \le A_{2}\y + \b_{2}$ is empty, then $R$ is simply described by $A_{3}\y \le \b_{3}$.
Also in this case, the preimage $\proj^{-1}(\y)$ of every point $\y \in R$ is infinite.
By the argument in Lemma~\ref{lem:big_fiber} below, we have a simple description $\proj(Q \cap \pat L)\. = \. R \cap \pat L'$, which finishes the proof.
So now assume that the two systems $A_{1}\y + \b_{1} \le x_{1}$ and $x_{1} \le A_{2}\y + \b_{2}$ are both non-empty.
Then we can decompose
\begin{equation}\label{eq:FM_decomp}
R \. = \. \bigsqcup_{j=1}^{r} P_{j},
\end{equation}
where each $P_{j}$ is a copolyhedron, so that over each $P_{j}$, the largest entry in the vector $A_{1} \y + \b_{1}$ is $\a_{j1} \y + b_{j1}$ and the smallest entry in the vector $A_{2} \y + \b_{2}$ is $\a_{j2}\y + b_{j2}$.
Thus, for every $\y \in P_{j}$, we have $\iproj(\y) = [\al_{j}(\y),\be_{j}(\y)]$, where $\al_{j}(\y) = \a_{j1} \y + b_{j1}$ and $\be_{j}(\y) = \a_{j2}\y + b_{j2}$ are affine rational functions.
Let $m = \eta(Q)$.
Note that the system $C\y \le \d$ describing $R$ contains at most $O(m^{2})$ inequalities, i.e., $\eta(R) = O(m^{2})$.
Also, we have $r = O(m^{2})$ and $\eta(P_{j}) = O(m)$ for $1 \le j \le r$.

For each $\y \in R$, the preimage $\iproj(\y) \subseteq Q$ is a segment in the direction $x_{1}$. Denote by $|\iproj(\y)|$ the length of this segment.
Now we refine the decomposition in \eqref{eq:FM_decomp} to
\begin{equation}\label{eq:refined_decomp}
R \. = \. R_{0} \sqcup R_{1} \sqcup \dots \sqcup R_{r}\., \quad \ \text{where}
\end{equation}

\begin{enumerate}[label=\alph*)]
\item Each $R_{j}$ is a copolyhedron in $\R^{n}$, with $\eta(R_{j}) = O(m^{2})$ and $r = O(m^{2})$.
\item For every $\y \in R_{0}$, we have the length $|\iproj(\y)| \ge \ell$.
\item For every $\y \in R_{j}$ ($1 \le j \le r$), we have the length $|\iproj(\y)| < \ell$. Furthermore, we have $\iproj(\y) = [\al_{j}(\y),\be_{j}(\y)]$, where $\al_{j}$ and $\be_{j}$ are affine rational functions in $\y$.
\end{enumerate}
This refinement can be obtained as follows. First, define
\begin{equation*}
R_{0} = \proj [ Q \cap (Q + \ell \v_{1}) ] \subseteq R,
\end{equation*}
where $\v_{1} = (1,0,\dots,0)$.
The facets of $R_{0}$ can be found from those of $Q \cap (Q + \ell \v_{1})$ again by Fourier--Motzkin elimination, and also $\eta(R_{0}) = O(m^{2})$.
Observe that $|\iproj(\y)| \ge \ell$ if and only if $\y \in R_{0}$.
Define $R_{j} \coloneqq P_{j} \cpl R_{0}$ for $1 \le j \le r$.
Recall that for every $\y \in P_{j}$, we have $\iproj(\y) = [\al_{j}(\y),\be_{j}(\y)]$. Therefore,
\begin{equation*}
R_{j} \. = \. P_{j} \cpl R_{0} \, = \, \{\y \in P_{j} : |\iproj(\y)| < \ell\}  \, = \, \{\y \in P_{j} : \al_{j}(\y) + \ell > \be_{j}(\y)\} \ts .
\end{equation*}
It is clear that each $R_{j}$ is a copolyhedron satisfying condition c).
Moreover, for each $1 \le j \le r$, we have $\eta(R_{j}) \le \eta(P_{j}) + 1 = O(m)$.
By \eqref{eq:FM_decomp}, we can decompose:
\begin{equation*}
R \. = \. R_{0} \ts\sqcup\ts (R \cpl R_{0}) \. = \. R_{0} \; \sqcup \; \bigsqcup_{j=1}^{r} (P_{j} \cpl R_{0}) \, = \, \bigsqcup_{j=0}^{r} \ts R_{j}\..
\end{equation*}
This decomposition satisfies all conditions a)--c) and proves \eqref{eq:refined_decomp}.
Note also that by converting each $R_{j}$ to $\floor{R_{j}}$, we do not lose any integer points in $R$.
Let us show that the part of $\proj(Q \cap \pat L)$ within $R_{0}$ has a simple pattern:

\begin{lemma}\label{lem:big_fiber}
$\proj(Q \cap \pat L) \cap R_{0} \. = \. R_{0} \cap \pat L'$.
\end{lemma}
\begin{proof}
Recall that $\proj(\pat L) = \pat L'$, which implies LHS $\subseteq$ RHS.
On the other hand, for every $\y \in \pat L'$, there exists $\x \in \pat L$ such that $\y=\proj(\x)$.
If $\y \in R_{0} \cap \pat L'$, we also have $|\iproj(\y)| \ge \ell$ by condition b), with $\ell$ defined in \eqref{eq:ldef}.
The point $\x$ and the segment $\iproj(\y)$ lie on the same vertical line.
Therefore, since $|\iproj(\y)| \ge \ell$, we can find another $\x'$ such that $\x' \in \iproj(\y) \subseteq Q$ and also $\x' - \x \in \L$.
Since $\pat L$ has period $\L$, we have $\x' \in \pat L$.
This implies $\x' \in Q \cap \pat L$, and $\y \in \proj(Q \cap \pat L)$.
Therefore we have RHS $\subseteq$ LHS, and
the lemma holds.
\end{proof}

It remains to show that $\. \proj(Q \cap \pat L) \cap R_{j} \.$ also has a pattern for every~$j>0$.
By condition c), every such $R_{j}$ has a ``thin'' preimage.
Let $Q_{j} = \iproj(R_{j}) \subseteq Q$.
If $\affdim(R_{j}) < n$, we have $\affdim(Q_{j}) < n+1$.
In this case we can apply the inductive hypothesis.
Otherwise, assume $\affdim(R_{j}) = n$.
For convenience, we refer to $R_{j}$ and $Q_{j}$ as just $R$ and $Q$.
We can write $R = R' + D$, where $R' \subseteq R$ is a polytope and $D$ is the recession cone of $R$.

Consider $\y \in R$, $\mathbf{v} \in D$ and $\la > 0$. Since $\y + \la \mathbf{v} \in R$, from c) we have $\iproj(\y + \la \mathbf{v}) = [\al(\y + \la \mathbf{v}),\be(\y + \la \mathbf{v})]$.
Denote by $\wt \al$ and $\wt \be$ the linear parts of the affine maps $\al$ and $\be$.
By a property of affine maps, we have:
\begin{equation}\label{eq:lift_cone_1}
\iproj(\y+\la\mathbf{v}) = [\al(\y + \la \mathbf{v}),\be(\y + \la \mathbf{v})] = [\al(\y) + \la \wt \al(\mathbf{v}),\; \be(\y) + \la \wt \be(\mathbf{v})].
\end{equation}
Therefore,
\[
|\iproj(\y + \la \mathbf{v})| = \be(\y) - \al(\y) + \la \bigl(\wt \be - \wt \al\bigr) (\mathbf{v}).
\]
Since $(\y + \la \mathbf{v}) \in R$, by c) we have:
\[
0 \le |\iproj(\y + \la \mathbf{v})| = \be(\y) - \al(\y) + \la \bigl(\wt \be - \wt \al\bigr) (\mathbf{v}) < \ell.
\]
Because $\la > 0$ is arbitrary, we must have $\bigl(\wt \be - \wt \al\bigr)(\mathbf{v}) = 0$.
This holds for all $\mathbf{v} \in D$.
We conclude that $\wt\be - \wt\al$ vanishes on the whole subspace $H \coloneqq \span(D)$, i.e., for any $\mathbf{v} \in H$ we have $\wt\al(\mathbf{v}) = \wt\be(\mathbf{v})$.
Thus, we can rewrite \eqref{eq:lift_cone_1} as
\begin{equation}\label{eq:lift_cone_2}
\iproj(\y + \la \mathbf{v}) = [\al(\y),\be(\y)] + \la \wt\al(\mathbf{v}) = \iproj(\y) + \la \wt\al(\mathbf{v}).
\end{equation}

Define $C \coloneqq \wt\al(D)$ and $G \coloneqq  \wt \al (H)$.
Note that $\span(C) = G$, because $\span(D) = H$.
Recall that $R = R' + D$ with $R'$ a polytope.
In \eqref{eq:lift_cone_2}, we let $\y$ vary over $R'$, $\lambda$ vary over $\R_{+}$ and $\mathbf{v}$ vary over $D$.
The LHS becomes $Q = \iproj(R)$.
The RHS becomes $\iproj(R')  + C$.
Therefore, we have $Q = \iproj(R') + C$.
Since $\iproj(R')$ is a polytope, we conclude that $C$ is the recession cone for $Q$.

Because $\iproj(\y) = [\al(\y), \be(\y)]$ for every $\y \in R$, the last $n$ coordinates in $\al(\y)$ and $\be(\y)$ are equal to $\y$.
This also holds for $\wt\al(\y)$ and $\wt\be(\y)$, i.e., $\proj(\wt\al(\y)) = \proj(\wt\be(\y)) = \y$.
This implies $\proj(G) = H$, because $G = \wt\al(H)$.
In other words, $\wt\al$ is the inverse map for $\proj$ on $G$ (see Fig.~\ref{pic:proj}).


\begin{figure}[!h]
\centering
\begin{tikzpicture}[scale=1.5]
\draw[-,line width = 0.8] (-0.8,1.5)--(0.2,0)--(1,0)--(2,1.0);
\node[above] (R) at (0.2,0.2) {$R$};

\draw[-,line width = 1, blue] (0.2,0)--(1,0);

\fill[gray!35] (0.2,1.0)--(0.6,0.4)--(1.23,1.0);
\draw[dashed,->] (0.6,0.4)--(0.2,1.0);
\draw[dashed,->] (0.6,0.4)--(1.23,1.0);
\node[above] (D) at (0.66,0.6) {$D,H$};

\draw[-,line width = 0.8] (1.5,1.6)--(1,1.4)--(0.2,1.5)--(-0.5,2.2);
\draw[-,line width = 0.8] (1.5,1.4)--(1,1.2)--(0.2,1.3)--(-0.5,2.0);
\node[above] (Q) at (0.2,1.65) {$Q$};

\fill[blue!20] (0.2,1.5)--(0.2,1.3)--(1,1.2)--(1,1.4);
\draw[-,line width = 0.8, blue] (0.2,1.5)--(0.2,1.3);
\draw[-,line width = 0.8, blue] (1,1.4)--(1,1.2);
\draw[-,line width = 1,blue] (1,1.4)--(0.2,1.5);
\draw[-,line width = 1,blue] (1,1.2)--(0.2,1.3);

\fill[gray!35] (0.15,2.28)--(0.6,1.8)--(1.4,2.09);
\draw[dashed,->] (0.6,1.8)--(0.15,2.28);
\draw[dashed,->] (0.6,1.8)--(1.4,2.09);
\node[above] (C) at (0.7,1.85) {$C,G$};

\draw[->] (-0.1,1.5)--(-0.1,1.1);
\node[left] (proj) at (-0.1,1.3) {$\proj$};

\draw[->] (1.7,1.1)--(1.7,1.5);
\node[right] (al) at (1.7,1.3) {$\wt\al$};

\end{tikzpicture}
\caption{$R$ and $Q = \iproj(R)$, with $R'$ and $\iproj(R')$ shown in blue. The cones $C$ and $D$ span $G$ and $H$, respectively.}
\label{pic:proj}%
\end{figure}

Recall that $Q \cap \pat L$ is a patterned polyhedron with period $\L$, and $\proj(Q) = R$.
Define
\[
\LS \coloneqq \L \cap G \quad \text{and} \quad \T \coloneqq \proj(\LS) \subset \proj(G) = H.
\]
Since $\L$ is full-rank, we have $\rank(\LS) = \dim(G)$.
Since $\wt\al$ and $\proj$ are inverse maps, we have $\LS = \wt\al(\T)$.
We claim that $\proj(Q \cap \pat L) \subset R$ is a patterned polyhedron with period $\.\T$.
Indeed, consider any two points $\, \y_{1},\y_{2} \in R \,$ with $\y_{2} - \y_{1} \in \T$.
Assume that $\y_{1} \in \proj(Q \cap \pat L)$, i.e., there exists $\x_{1} \in Q \cap \pat L$ with $\proj(\x_{1}) = \y_{1}$.
We show that $\y_{2} \in \proj(Q \cap \pat L)$.
First, we have $\iproj(\y_1) = [\al(\y_{1}),\be(\y_{1})]$ and $\iproj(\y_2) = [\al(\y_{2}),\be(\y_{2})]$.
Let $\mathbf{v} = \y_{2} - \y_{1} \in \T \subset H$.
Since $\y_{2}=\y_{1}+\mathbf{v}$, we can apply~\eqref{eq:lift_cone_2} with $\lambda=1$ and get:
\begin{equation}
[\al(\y_{2}),\be(\y_{2})] = \iproj(\y_{2}) = \iproj(\y_{1} + \mathbf{v}) = [\al(\y_{1}),\be(\y_{1})] + \wt\al(\mathbf{v}).
\end{equation}
Thus, we have $\al(\y_{1}) - \be(\y_{1}) = \al(\y_{2}) - \be(\y_{2})$.
In other words, the points $\al(\y_{1}), \be(\y_{1}), \al(\y_{2})$ and $\be(\y_{2})$ form a parallelogram inside $Q$.
Since $\proj(\x_{1}) = \y_{1}$, we have:
\begin{equation*}
\x_{1} \in \iproj(\y_{1}) = [\al(\y_{1}),\be(\y_{1})] \subseteq Q.
\end{equation*}
So $\x_{1}$ lies on the edge $[\al(\y_{1}),\be(\y_{1})]$ of the parallelogram mentioned above.
Therefore, we can find another point $\x_{2}$ lying on the other edge $[\al(\y_{2}),\be(\y_{2})] = \iproj(\y_{2})$ with
\[
\x_{2} - \x_{1} = \al(\y_{2}) - \al(\y_{1}) = \wt \al(\y_{2} - \y_{1}) = \wt \al (\mathbf{v}) \in \wt \al(\T) = \LS.
\]
This $\x_{2}$ satisfies $\proj(\x_{2}) = \y_{2}$.
Recall that $\x_{1} \in \pat L$, with $\pat L$ having period $\L$.
Since $\x_{2} - \x_{1} \in \LS \subset \L$, we have $\x_{2} \in \pat L$.
This implies $\x_{2} \in Q \cap \pat L$ and $\y_{2} \in \proj(Q \cap \pat L)$.

So we have established that $\proj(Q \cap \pat L) \subset R$ is a patterned polyhedron with period $\T$.
Note that
$$\rank(\T)  = \rank(\LS) = \dim(G) = \dim(H) = \dim(D).$$
If $\dim(D) = n$ then $\T$ is full-rank.
If $\dim(D) < n$, recall that $R = R' + D$ where $R'$ is a polytope, and $\span(D) = H$.
Let $H^{\perp}$ be the complement subspace to $H$ in $\R^{n}$, and $R^{\perp}$ be the projection of $R'$ onto $H^{\perp}$.
Since $R^{\perp}$ is bounded, we can take a large enough lattice $\T^{\perp} \subset H^{\perp}$ such that there are no two points $\z_{1} \neq \z_{2} \in R^{\perp}$ with $\z_{1} - \z_{2} \in \T^{\perp}$.
Now the lattice $\T^{\perp} \oplus \T$ is full-rank, which can be taken as a period for $\proj(Q \cap \pat L)$.

To summarize, for every piece $R_{j}$ and $Q_{j} = \proj^{-1}(R_{j})$, $1 \le j \le r$, the projection $\proj(Q_j \cap \pat L) \subset R_{j}$ has period~$\T_{j}$. Thus $\proj(Q_{j} \cap \pat L)$ is a patterned polyhedron.
This completes the proof.
\ $\sq$

\bigskip

\section{Finding short GF for unbounded projection}\label{sec:GF}

\subsection{Barvinok--Woods algorithm}

In this section, we are again assuming that dimensions $m$ and $n$ are fixed. We recall the Barvinok--Woods algorithm, which finds in polynomial time a short GF for the projection of integer points in a polytope:

\begin{theo}[\cite{BW}]\label{th:BW}
Let $m, n \in \nn$ be fixed dimensions.
Given a rational polytope $Q = \{\x \in \R^{m} : A\x \le \b\}$,
and a linear transformation $T : \rr^{m} \to \rr^{n}$ represented as a matrix $T \in \zz^{n \times m}$,
there is a polynomial time algorithm to compute a short GF for $T(Q \cap \Z^{m})$ as:
\begin{equation}\label{eq:shortGF}
g(\t) \, = \, \sum_{\y \; \in \; T(Q \cap \, \Z^{m})} \t^{\y}  \; = \; \sum_{i=1}^{M} \frac{c_{i} \. \t^{\a_{i}}}{(1-\t^{\b_{i1}}) \dots (1-\t^{\b_{i s}})},
\end{equation}
where $c_{i}=p_{i}/q_{i} \in \QQ ,\; \a_{i},\b_{ij} \in \Z^{n} ,\; \b_{ij} \neq 0$ for all $i,j$, and $s$ is a constant depending only on~$m$.
Furthermore, the short GF $g(\t)$ has length $\phi(g) = \polyin(\phi(Q) + \phi(T))$, where
\begin{equation}\label{GFsize}
\phi(g) \, = \,
\sum_{i} \. \lceil\log_2 |p_{i} \. q_{i}|+1\rceil \, + \,
\sum_{i, j} \. \lceil\log_2 a_{i\ts j}+1\rceil \, + \,
\sum_{i,j,k} \. \lceil\log_2 b_{i \ts j \ts k}+1\rceil\ts.
\end{equation}
\end{theo}

Clearly, our main result Theorem~\ref{th:main_2} is an extension
of Theorem~\ref{th:BW}.
The proof of Theorem~\ref{th:main_2} is based
on Theorem~\ref{th:main_1} and uses the following standard result:

\begin{prop}[see e.g.\ \cite{Mei}]\label{lem:triangulate}
Let $n\in \nn$ be fixed.
Let $R = \{\x \in \R^{n} : C\x \le \d\}$  be a possibly unbounded polyhedron.
There is a decomposition
\begin{equation}\label{eq:triangulate}
R = \bigsqcup_{k=1}^{t} \, R_{k} \oplus D_{k}\.,
\end{equation}
where each $R_{k}$ is a copolytope, and each $D_{k}$ is a simple cone.
Each part $R_{k} \oplus D_{k}$ is a direct sum, with $R_{k}$ and $D_{k}$ affinely independent.
All $R_{k}$ and $D_{k}$ can be found in time $\polyin(\phi(R))$.
\end{prop}

Before proving Theorem~\ref{th:main_2}, we make an important remark:

\begin{rem}  \label{r:convergence}
{\rm
The extra condition $T(Q) \subseteq \rr_{+}^{n}$ in Theorem~\ref{th:main_2} is to make sure that the power series $\sum \t^{\y}$ of $T(Q\cap\zz^{m})$ converges on a non-empty open domain to the computed short GF.
In general, without the condition $T(Q) \subseteq \rr_{+}^{n}$, we can still make sense of the infinite GF (see Section~\ref{ss:convergence2}).
}
\end{rem}

\subsection{Proof of Theorem~\ref{th:main_2}.}
WLOG, we can assume $\affdim(Q) = m$ and $\affdim(T(Q)) = n$.
Clearly, the set $X = Q \cap \Z^{m}$ is a semilinear set, and we want to find a short GF for $T(X)$.

First, we argue that for any bounded polytope $P \subset \R^{n}$, a short GF for $T(X) \cap P$ can be found in time $\polyin(\phi(Q) + \phi(P))$. Assume $P$ is given by a system $C\y \le \d$.
For any $\v \in P$, we have $\v \in T(X)$ if and only if the following system has a solution $\x \in \Z^{m}$:
\begin{equation}\label{eq:IP}
\Bigl\{
\begin{matrix}
A\x &\le  &\b \\
T(\x) &=  &\v
\end{matrix} \;.
\end{equation}
By a well known bound on integer programming solutions (see  \cite[Cor.~17.1b]{Schrijver}), it is equivalent to find such a solution $\x$ with length at most a polynomial in the length of the system \eqref{eq:IP}.
The parameter $\v$ lies in $P$, which is bounded.
Therefore, we can find a number $N$ of binary length $\phi(N) = \polyin(\phi(P) + \phi(Q))$, such that \eqref{eq:IP} is equivalent to:
\begin{equation*}
\Biggl\{
\begin{matrix}
A\x &\le &\b \\
C\.T(\x) &\le &\d \\
-N \le \x \hspace{-.1em} &\le &N
\end{matrix} \;.
\end{equation*}
This system describes a polytope $\wh Q \subset \R^{m}$.
Applying Theorem~\ref{th:BW} to $\wh Q$, we obtain a short GF $g(\t)$ for $T(\wh Q \cap \Z^{m}) = T(X) \cap P$.

Now we are back to finding a short GF for the entire projection $T(X)$.
Applying Theorem~\ref{th:main_1} to $X$, we have a decomposition:
\begin{equation}\label{eq:BW_decomp}
T(X) = \bigsqcup_{j=1}^{r} \,R_{j} \cap \pat T_{j}\..
\end{equation}

We proceed to find a short GF $g_{j}$ for each patterned polyhedron $R_{j} \cap \pat T_{j}$ with period $\T_{j}$.
For convenience, we refer to $R_{j},\. \pat T_{j},\. \T_{j},\. g_{j}$ simply as $R ,\.\pat T ,\. \T$ and $g$.
By Proposition~\ref{lem:triangulate}, we can decompose
\begin{equation}\label{eq:direct_decomp}
R \.= \. \bigsqcup_{i=1}^{t_{j}} \, R_{i} \oplus D_{i} \, \quad \text{and} \quad R \cap \pat T \, = \,
\bigsqcup_{i=1}^{t_{j}} \, (R_{i} \oplus D_{i}) \cap \pat T \ts.
\end{equation}
Recall from Theorem~\ref{th:main_1} that $\T$ has full rank.
Let $d_{i} = \dim(D_{i})$ and $\v_{i}^{1},\dots,\v_{i}^{d_{i}}$ be the generating rays of the (simple) cone $D_{i}$.
For each $\v_{i}^{t}$, we can find $n_{t} \in \Z_{+}$ such that $\w_{i}^{t} = n_{t} \v_{i}^{t}  \in \T$.
Let $P_{i}$ and $\T_{i}$ be the parallelepiped and lattice spanned by $\w_{i}^{1},\dots,\w_{i}^{d_{i}}$, respectively.
We have $D_{i} = P_{i} + \T_{i}$ and therefore
\begin{equation}\label{eq:periodize}
R_{i} \oplus D_{i} = R_{i} \oplus (P_{i} + \T_{i}) = (R_{i} \oplus P_{i}) + \T_{i}.
\end{equation}
Each $R_{i} \oplus P_{i}$ is a copolytope.
Note that Theorem~\ref{th:BW} is stated for (closed) polytopes.
We round each $R_{i} \oplus P_{i}$ to $\floor{R_{i} \oplus P_{i}}$, where $\floor{.}$ was described in Definition~\ref{def:copoly} (Section \ref{sec:short_proof}).
By the earlier argument, we can find a short GF $h_{i}(\t)$ for $T(X) \cap (R_{i} \oplus P_{i}) = (R_{i} \oplus P_{i}) \cap \pat T$.
Since $\T_{i} \subseteq \T$, the pattern $\pat T$ also has period $\T_{i}$.
By \eqref{eq:periodize}, we can get the short GF $f_{i}(\t)$ for $(R_{i} \oplus D_{i}) \cap \pat T$ as
\begin{equation}\label{eq:periodize_GF}
f_{i}(\t) = \sum_{\y \in (R_{i} \oplus D_{i}) \cap \pat T} \t^{\y}
= \Bigg(\sum_{\y \in (R_{i} \oplus P_{i}) \cap \pat T}  \t^{\y} \Bigg) \cdot \Bigg( \sum_{\y \in \T_{i}} \t^{\y} \Bigg)
= h_{i}(\t) \; \prod_{t=1}^{d_{i}} \frac{1}{1 - \t^{\w_{i}^{t}}}.
\end{equation}
By \eqref{eq:direct_decomp}, we obtain
\begin{equation}\label{eq:sum_up}
g(\t) = \sum_{\y \in R \cap \pat T} \t^{\y} \,
= \, \sum_{1 \. \le \. i \. \le \. t_{j}} f_{i}(\t).
\end{equation}

In summary, we obtained a short GF $g_{j}(\t)$ for each piece $R_{j} \cap \pat T_{j} \; (1 \le j \le r)$.
Summing over all $\. j \.$ in~\eqref{eq:BW_decomp}, we get a short GF for $T(X)$, as desired.
\ $\sq$

\bigskip

\section{Sets defined by Presburger formulas}

Now we employ Theorem~\ref{th:main_1} to analyze the structure of general semilinear sets.
For our purpose, these are best defined in the context of \emph{Presburger Arithmetic} ($\Pr$).
In this section, all variables $x,y,z,\x,\y,\z$, etc., are over $\Z$.
$\Pr$ is the first order theory on the integers that allows only additions and inequalities. In other words, each \emph{atom} (quantifier and Boolean free term) in $\Pr$ is an integer inequality of the form
\[
a_{1} x_{1} + \ldots + a_{n} x_{n} \le b,
\]
where $\x = (x_{1},\dots,x_{n})$ are integer variables, and $a_{1},\dots,a_{n},b \in \Z$ are integer constants. A \emph{$\Pr$-formula} is formed by taking Boolean combinations (negations, conjunctions, disjunctions) of such inequalities, and also applying quantifiers $\for/\ex$ over some of the variables.
A \emph{$\Pr$-sentence} is a $\Pr$-formula with all variables quantified.
For instance, an integer programming problem $\ex \x : A\x \le \b$ is an existential $\Pr$-sentence with only conjunctions.

Fix $k \in \Z_{+}$ and a vector of dimensions $\n = (n_{1},\dots,n_{k}) \in \Z_{+}^{k}$.
Let $\x_{1} \in \Z^{n_{1}}, \dots, \x_{k} \in \Z^{n_{k}}$ be vectors of integer variables.
We consider the class $\Pr_{k,\n}$ consisting of $\Pr$-formulas of the form:
\begin{equation*}
(*) \qquad F \; = \;
\bigl\{ \x_{1} : Q_{2}\ts\x_{2} \dots Q_{k}\ts\x_{k} \; \Phi(\x_{1}, \dots, \x_{k}) \bigr\}.
\end{equation*}
Here~$Q_{2},\dots,Q_{k} \in \{\for,\ex\}$ are any~$k$ quantifiers, and~$\Phi(\x_{1}, \dots, \x_{k})$ is a Boolean combination of linear inequalities in~$\x_{1},\dots,\x_{k}$.
For a specific value of~$\x_{1} \in \Z^{n_{1}}$, the \emph{substituted formula}~$F(\x_{1})$ is a $\Pr$-sentence in variables~$\x_{2}, \dots, \x_{k}$.
We say that~$\x_{1}$ \emph{satisfies}~$F$ if~$F(\x_{1})$ is a true $\Pr$-sentence.
To simplify the notation, we identify a $\Pr$-formula $F$ with the set of integer points $\x_{1}$ that satisfy $F$.
The length~$\phi(F)$ is the total length of all symbols and constants in~$F$ written in binary.

\begin{eg}
{\rm
The $\Pr$-formula~$F = \{x : \for y \; (5y \ge x+1) \lor (5y \le x - 1)\} \in \Pr_{2,(1,1)}$ determines the set of non-multiples of~$5$.
}
\end{eg}

By a classical result of Ginsburg and Spanier~\cite{GS}, semilinear sets (Definition~\ref{def:sem_lin}) are exactly those definable in $\Pr$, i.e., representable by a $\Pr$-formula $F$ of the form $(*)$ for some $k,\n$.
Below is our main result for this section, which generalizes Theorem~\ref{th:main_1}.
Roughly speaking, it allows us to compute in polynomial time the ``periods'' of a semilinear set when represented as a $\Pr$-formula:

\begin{theo}\label{th:pres}
Fix~$k$ and~$\n$. Given a $\Pr$-formula~$F \in \Pr_{k,\n}$, there exists a decomposition
\begin{equation*}
F = \bigsqcup_{j=1}^{r} \, R_{j} \cap \pat T_{j} \.,
\end{equation*}
where each~$R_{j} \cap \pat T_{j}$ is a patterned polyhedron in~$\R^{n_{1}}$ with period~$\T_{j} \subseteq \Z^{n_{1}}$. The polyhedra~$R_{j}$ and lattices~$\T_{j}$ can be found in time~$\polyin(\phi(F))$.
\end{theo}

\begin{proof}
Consider any~$F \in \Pr_{k,\n}$ of the form:
\begin{equation*}
F = \{ \x_{1} : Q_{2}\ts \x_{2} \dots Q_{k}\ts \x_{k} \; \Phi(\x_{1}, \dots, \x_{k}) \}.
\end{equation*}
Let~$\ov\x = (\x_{1},\dots,\x_{k})$ and~$n = n_{1} + \ldots + n_{k}$.
Let us show directly that
\begin{equation*}
X = \{ \ov\x \in \Z^{n} : \Phi(\ov\x)\}
\end{equation*}
is semilinear.
Recall that~$\Phi$ is a Boolean combination of linear inequalities.
Using Proposition~5.2.2 in~\cite{Woods}, we can rewrite~$\Phi$ into a disjunctive normal form of polynomial length:
\begin{equation*}
\Phi = (A_{1} \ov\x \le \b_{1}) \lor \ldots \lor (A_{t} \ov\x \le \b_{t}).
\end{equation*}
Here, each~$A_{i} \ov\x \le \b_{i}$ is a system of inequalities, describing a polyhedron~$P_{i} \subseteq \R^{n}$.
Moreover, all polyhedra~$P_{1},\dots,P_{t}$ are pairwise disjoint, and~$\sum_{i=1}^{t} \phi(P_{i}) = \polyin(\phi(F))$.
In other words, the set~$X$ consists of integer points in a disjoint union of~$\. t \.$ polyhedra.
Thus,~$X$ is a semilinear set with~$\psi(X) = \polyin(\phi(F))$, in the notation of Definition~\ref{def:sem_lin}.

The proof goes by recursive construction of sets $X^{(k)}, X^{(k-1)},\dots,X^{(1)}$.
Let~$X^{(k)} \coloneqq X$.
If~$Q_{k} = \ex$, we consider the set
\begin{equation*}
X^{(k-1)} \. \coloneqq \. \bigl\{(\x_{1}, \dots, \x_{k-1}) : \ex \x_{k} \; \Phi(\ov\x)\bigr\} \. = \. \bigl\{(\x_{1},\dots, \x_{k-1}) : \ex \x_{k} \;  [\ov\x \in X^{(k)}] \bigr\}.
\end{equation*}
This set~$X^{(k-1)}$ is obtained from~$X^{(k)}$ by projecting along the last variable~$\x_{k}$, i.e., the last~$n_{k}$ coordinates in~$\ov\x$.
By Theorem~\ref{th:main_1}, we can find in polynomial time a decomposition of the form \eqref{eq:decomp_many} for~$X^{(k-1)}$.
Moreover, we have~$\psi(X^{(k-1)}) = \polyin(\psi(X^{(k)}))$.

Similarly, if~$Q_{k} = \for$, we consider
\begin{equation*}
X^{(k-1)} \. \coloneqq \. \bigl\{(\x_{1}, \dots, \x_{k-1}) : \for \x_{k} \; \Phi(\ov\x)\bigr\} \. =
\. \lnot \bigl\{(\x_{1},\dots, \x_{k-1}) : \ex \x_{k} \;  [\ov\x \in \lnot X^{(k)}]\bigr\}.
\end{equation*}
Here~$\lnot$ denotes the complement of a set.
Observe that the complement~$\lnot X$ of a semilinear set~$X$ is also semilinear, and~$\psi(\lnot X) = \polyin(\psi(X))$.
Indeed, assume that~$X$ has a decomposition
\[
X = \bigsqcup_{i=1}^{p} \, P_{i} \cap \pat L_{i} \. .
\]
Recall that the polyhedral pieces~$P_{i}$ are pairwise disjoint, but do not necessarily cover~$\R^{n}$.

Let us prove that the complement~$\bigl ( \R^{n} \cpl \bigsqcup_{i=1}^{p} P_{i} \bigr )$ can also be partitioned into polynomially many pairwise disjoint polyhedra.
Indeed, we can represent $\bigsqcup_{i=1}^{p} P_{i}$ by a Boolean expression of linear inequalities in $\x$.
Therefore, the complement can also be represented by a Boolean expression.
By Proposition~5.2.2 in~\cite{Woods} mentioned above, we can rewrite the complement as a disjoint union of polynomially many polyhedra~$P'_{1}, \dots, {P'_{q}}$.
From here, we obtain the decomposition:
\[
\lnot X = \bigsqcup_{i=1}^{p} \, P_{i} \cap \pat L'_{i} \; \sqcup \; \bigsqcup_{j=1}^{q} \, P'_{j} \cap \Z^{n} \.,
\]
where~$\pat L'_{i}$ is the complement of~$\pat L_{i}$, with the same period~$\L_{i}$.
Therefore, we have~$\psi(\lnot X^{(k)}) = \polyin(\psi(X^{(k)}))$.
Applying Theorem~\ref{th:main_1}, we can obtain~$X^{(k-1)}$ by projecting~$\lnot X^{(k)}$.

Applying the above argument recursively for quantifiers~$Q_{k-1},\dots,Q_{2}$, we obtain a polynomial length decomposition for the semilinear set
\begin{equation*}
X^{(1)} = \{\x_{1} \in \Z^{n_{1}} : Q_{2}\ts \x_{2} \dots Q_{k}\ts \x_{k} \; \Phi(\x) \} = F.
\end{equation*}
This completes the proof.
\end{proof}

\begin{theo}\label{th:pres_GF}
Fix~$k$ and~$\n$. Let~$F \in \Pr_{k,\n}$ be a $\Pr$-formula and $M$ be a positive integer. Denote by $f_{M}(\t)$ the partial GF
\begin{equation}\label{eq:partial_GF}
f_{M}(\t) \coloneqq \sum_{\x \in F \cap [-M, M]^{n_{1}}} \t^{\x}.
\end{equation}
Suppose there is an oracle computing $f_{M}(\t)$ as a short GF $(\divideontimes)$ in time $\mu(F,M)$.
Then there is an integer $N = N(F)$ with $\log N = \polyin(\phi(F))$, such that the GF~$f(\t) = \sum_{\x \in F} \t^{\x}$ for the entire set~$F$ can be computed as a short GF in time $\polyin(\mu(F,N))$.
The integer $N = N(F)$ can be computed in time $\polyin(\phi(F))$.
\end{theo}

In other words, Theorem \ref{th:pres_GF} says that the full GF $f(\t)$ can be computed in polynomial time from the partial GF $f_{N}(\t)$ for a suitable $N$.

\begin{proof}
Let~$n = n_{1}$. By Theorem~\ref{th:pres}, we have a decomposition
\begin{equation*}
F = \bigsqcup_{j=1}^{r} \, R_{j} \cap \pat T_{j} \..
\end{equation*}
We proceed similarly to the proof of Theorem~\ref{th:main_2}.
Denote~$R_{j}$ and~$\pat T_{j}$ by~$R$ and~$\pat T$ respectively, for convenience.
We have the decomposition \eqref{eq:direct_decomp} for~$R$ and~$R \cap \pat T$,
which leads to \eqref{eq:periodize}.
Eventually, we can compute a short GF~$g(\t)$ for~$R \cap \pat T$ using \eqref{eq:periodize_GF} and \eqref{eq:sum_up}.
The only difference is that the GF~$h_{i}$ for each patterned polytope~$(R_{i} \oplus P_{i}) \cap F$, which was~$(R_{i} \oplus P_{i}) \cap \pat T$ in \eqref{eq:periodize_GF}, cannot be obtained from Theorem~\ref{th:BW}, since~$F$ is no longer the result of a single projection on a polyhedron.

Recall that each~$R_{i} \oplus P_{i}$ is a polytope, with facets of total length~$\polyin(\phi(F))$.
Therefore, the vertices of~$R_{i} \oplus P_{i}$ can be found in polynomial time given~$F$.
This holds for all~$1 \le i \le t_{j}$ and all~$1 \le j \le r$.
Thus, we can find a positive integer~$N = N(F)$, for which
\begin{equation*}
\log N = \polyin(\phi(F)) \quad \text{and} \quad R_{i} \oplus P_{i} \subseteq [-N,N]^{n} \quad \text{ for all } 1 \le i \le t_{j}.
\end{equation*}
Given the partial GF~$f_{N}(\t)$, the GF~$h_{i}(\t)$ for each~$(R_{i} \oplus P_{i}) \cap F$ can be computed as follows.

Barvinok's theorem \cite{B1} (see also Theorem~4.4 in \cite{BP}) allows us to compute in polynomial time a short GF
\[
g_{i}(\t) = \sum_{\x \in (R_{i} \oplus P_{i}) \cap \Z^{n}} \t^{\x}
\]
for each polytope~$R_{i} \oplus P_{i}$.
Theorem~10.2 in \cite{BP} allows us to compute in polynomial time a short GF for the intersection of two finite sets, given their short GFs as input.
Since $(R_{i} \oplus P_{i}) \cap F$ is the intersection of $(R_{i} \oplus P_{i}) \cap \Z^{n}$ and $F \cap [-N,N]^{n}$, we can compute
\[
h_{i}(\t) \, = \sum_{\x \in (R_{i} \oplus P_{i}) \cap F} \t^{\x} \, = \. \Biggl(\sum_{\x \in (R_{i} \oplus P_{i}) \cap \Z^{n}} \t^{\x}\Biggr) \star \Biggl(\sum_{\x \in F \cap [-N,N]^{n}} \t^{\x}\Biggr) \. = \, g_{i}(\t) \star f_{N}(\t).
\]
in time $\polyin(\mu(F,N))$.
Here~$\star$ is the \emph{Hadamard product} of two power series (see~\cite{BP}).
The short GF $f_{N}(\t)$ is obtained by a single call to the oracle in time $\mu(F,N)$.
This completes the proof.
\end{proof}

\begin{rem}\label{rem:NP-hard}
We emphasize that Theorem~\ref{th:pres_GF} does not directly compute the GF~$f(\t)$ in polynomial time, for a general~$F$.
It only claims that~$f(\t)$ can be computed in time $\polyin\bigl(\mu(F,N)\bigr)$ given the oracle.
In fact, computing~$f(\t)$ directly from~$F$ is an $\NP$-hard problem, even for~$F \in \Pr_{2,(1,1)}$.
This result is proved in \cite[Prop.~5.3.2]{Woods}, and is ultimately derived from a result by Sch\"{o}ning~\cite{S},
which says that deciding the truth of $\Pr$-sentences of the form $\ex x \for y \; \Phi(x,y)$
is an $\NP$-complete problem.
\end{rem}


\bigskip

\section{The~$k$-feasibility problem} \label{s:k-feas}

We present an application of Theorem~\ref{th:pres_GF}.
Let~$n,d$ and~$k$ be fixed integers and~$A \in \Z^{d \times n}$.
In \cite{ADL}, the authors defined a set~$\sg_{\ge k}(A) \in \Z^{d}$ of \emph{$k$-feasible} vectors as
\begin{equation}\label{eq:sgk}
\sg_{\ge k}(A) = \{\y  \in \Z^{d} \. : \. \ex \.  \x_{1}, \dots,  \x_{k} \in \N^{n}, \; \y = A  \x_{j}, \;  \x_{i} \neq  \x_{j} \text{ if } i \neq j ,\; 1 \le i,j \le k \}.
\end{equation}
In other words,~$\sg_{\ge k}(A)$ consists of vectors that are representable in at least~$k$ different ways as a non-negative combination of columns of~$A$.
In addition to some results about~$\sg_{\ge k}(A)$, the authors also gave an algorithm to compute a short GF for~$\sg_{\ge k}(A)$ within a finite box:

\begin{theo}[Theorem~5 in \cite{ADL}]\label{th:ADL}
Fix~$n,d$ and~$k$.
Let~$A \in \Z^{d \times n}$, and let~$N$ be a positive integer.
Let
\begin{equation*}
f_{N}(\t) = \sum_{\x \. \in \. \sg_{\ge k}(A) \cap [-N,N]^{d}} \t^{\x}
\end{equation*}
be the partial GF for $\sg_{\ge k}(A)$ within the box~$[-N,N]^{d}$.
Then there is a polynomial time algorithm to compute $f_{N}(\t)$ as a short GF.
\end{theo}

Using Theorem~\ref{th:pres_GF}, we can extend Theorem~\ref{th:ADL} as follows:

\begin{theo}\label{th:ADL_unbounded}
Fix~$n,d$ and~$k$.
Then there is a polynomial time algorithm to compute
\[
f(\t) = \sum_{\x \. \in \. \sg_{\ge k}(A)} \t^{\x}
\]
for the entire set~$\sg_{\ge k}(A)$, as a short GF.
\end{theo}

\begin{proof}
From the definition \eqref{eq:sgk}, we see that~$\sg_{\ge k}(A)$ is a $\Pr$-formula in variables~$\y, \x_{1},\dots,\x_{k}$ with only an existential ($\ex$) quantifier.
Indeed, each condition~$\y = A\x_{j}$ is a system of~$2d$ inequalities.
Each condition~$\x_{i} \neq \x_{j}$ is a disjunction of~$2n$ inequalities~$(x_{it} < x_{jt})$ or~$(x_{it} > x_{jt})$ for~$1 \le t \le n$.
Therefore, we have~$\sg_{\ge k}(A) \in \Pr_{k+1,\n}$, where~$\n = (d,n,\dots,n)$.

Applying Theorem~\ref{th:pres_GF}, we can compute in polynomial time a short GF~$f(\t)$ for~$\sg_{\ge k}(A)$ given the partial short GF~$f_{N}(\t)$.
Finally, Theorem~\ref{th:ADL} allows us to compute~$f_{N}(\t)$ in polynomial time.
\end{proof}

Theorem~\ref{th:ADL} was stated in \cite{ADL} for fixed~$n$ and~$k$, but arbitrary~$d$.
The following result is a straightforward consequence of the previous theorem and
an argument by P.~van~Emde~Boas described in~\cite[$\S$4]{L}.

\begin{theo}\label{th:ADL_general}
Fix $n$ and $k$, but let $d$ be arbitrary.
Then there is a polynomial time algorithm to compute
\[
f(\t) = \sum_{\x \. \in \. \sg_{\ge k}(A)} \t^{\x}
\]
for the entire set~$\sg_{\ge k}(A)$, as a short GF.
\end{theo}

\begin{proof}
This can be easily reduced to the case when~$d$ is also fixed.
Indeed, let~$\L_{A} \subseteq \Z^{d}$ be the lattice generated by the~$n$ columns of~$A \in \Z^{d \times n}$.
We have~$\rank(\L_{A}) = \rank(A) \le n$.
Hence, we can find a~$d \times d$ unimodular matrix~$U$ so that~$UA$ is non-zero only in the first~$n$ rows.
Let~$B \in \Z^{n \times n}$ be the first~$n$ rows of~$UA$, and~$\L_{B}$ be the lattice generated by the columns of~$B$.
Observe that~$\L_{B}$ and~$\L_{A}$ are isomorphic.
Therefore, the set of~$k$-representable vectors in~$\L_{A}$ are in bijection with those in~$\L_{B}$.
Now we apply Theorem~\ref{th:ADL_unbounded} to get a short GF $g(\t)$ for $\sg_{\ge k}(B)$.
The GF for $\sg_{\ge k}(A)$ is easily obtained from $g(\t)$ by a variable substitution via $U^{-1}$.
\end{proof}

\bigskip

%
%

\bigskip

\section{Conclusion and Final Remarks}\label{sec:Conclusion}

\subsection{}
We extend the Barvinok--Woods algorithm to compute short GFs for
projections of polyhedra.  The result fills a gap in the literature
on parametric integer programming which remained open since~2003.
We also prove a structural result on the projection of semilinear sets
by a direct argument. Let us emphasize that we get effective polynomial
bounds for the number of polyhedral pieces and the
facet complexity of each piece in the projection,
but not on the complexity of the pattern within each piece.

\subsection{}
The study of semilinear sets has numerous applications in computer science,
such as analysis of \emph{number decision diagrams} (see \cite{Leroux1}),
and \emph{context-free languages} (see \cite{Parikh}).
We refer to~\cite{Ginsburg} for background on semilinear sets with their connections to
Presburger Arithmetic, and to~\cite{CH} for most recent developments.
Let us also mention
that in the papers~\cite{NP1,NP2}, we analyze the semilinear structure of sets defined by \emph{short Presburger formulas}, which are $\Pr$-formulas with a bounded number of variables and inequalities.

\subsection{}\label{ss:convergence2}
Without the extra condition $T(Q) \subseteq \rr_{+}^{n}$ in Theorem~\ref{th:main_2} we can still treat the GF of $T(Q \cap \zz^{m})$
as formal power series.
In some cases, this power series might not converge under numerical substitutions.
For example, if $Q = \R^{m}$ and $T$ projects $\zz^{m}$ onto $\zz$, then every $y \in \Z$ lies in $T(Q \cap \zz^{m})$.
In this case, we have
\begin{equation*}
\sum_{y \ts\in\ts T(Q \cap \zz^{m})} t^{y} \;=\; \ldots + t^{-2} + t^{-1} + 1 + t + t^{2} + \ldots,
\end{equation*}
which is not convergent for any non-zero $t$.
However, when $T(Q)$ has a pointed characteristic cone, for example $T(Q) \subseteq \rr_{+}^{n}$, then the power series converges on a non-empty open domain.
For any $\t$ in that domain, the power series converges to the computed rational function $g(\t)$.
For the general case when $T(Q)$ could possibly contain infinite lines,
we can resort to the theory of valuations (see~\cite{B3,BP}) to make sense of the GF.
Alternatively, one can always decompose any such $Q$ into a finite union of at most $n+1$ polyhedra $Q_{i}$, each of which projects within a pointed cone in $\R^{n}$.
Then the GF for the projection of $Q \cap \Z^{m}$ can be thought of as a formal sum of at most $n+1$ short GFs, each with its own domain of convergence and a rational representation $g_{i}(\t)$.

\subsection{}
Our generalization of the Barvinok--Woods theorem also simplifies many existing proofs in the literature when one needs to compute a short generating function for unbounded sets.
See for example the computation of Hilbert series in \cite[Sec.\ 7.3]{BW} and the computation of optimal points for integer programming in~\cite[Lem.\ 3.3]{HS}.

\subsection{}
Finally, we refer to~\cite{RA} for an extensive introduction to the
Frobenius problem.  This problem was the first application of Kannan's pioneering result in~\cite{K2} on lattice covering radius, an application first suggested by
Lov\'{a}sz~\cite{Lov}.


\medskip

\subsection*{Acknowledgements}
We are greatly indebted to Sasha Barvinok and Sinai Robins
for introducing us to the subject. We are also thankful to
Iskander Aliev, Matthias Aschenbrenner, Art\"{e}m Chernikov,
Jes\'{u}s De Loera, Lenny Fukshansky, Oleg Karpenkov and Kevin Woods for interesting
conversations and helpful remarks.
The second author was partially supported by the~NSF.

\bigskip

\bigskip

\medskip


\end{document}